\documentclass[a4paper,11pt]{amsart}

\usepackage{amsfonts, latexsym, amsmath, amssymb, amsthm, amscd}
\usepackage[all]{xy}
\usepackage[english]{babel}
\usepackage[utf8]{inputenc}
\usepackage{graphicx}
\usepackage{booktabs}
\usepackage{array, tabularx}
\usepackage{textcomp}
\usepackage{tikz, tikz-cd}
\usepackage{multirow}
\usepackage{comment}
\usepackage{url}
\usepackage{tensor}
\usepackage{MnSymbol}
\usepackage{color}
\usepackage[hypertexnames=false,
backref=page,
    pdftex,
    pdfpagemode=UseNone,
    breaklinks=true,
    extension=pdf,
    colorlinks=true,
    linkcolor=blue,
    citecolor=blue,
    urlcolor=blue,
]{hyperref}

\usepackage{enumitem}
\usepackage[top=1.5in, bottom=.9in, left=0.9in, right=0.9in]{geometry}

\theoremstyle{plain}
\newtheorem{theorem}{Theorem}[section]
\newtheorem*{theorem*}{Theorem}
\newtheorem{definition}[theorem]{Definition}
\newtheorem{lemma}[theorem]{Lemma}

\newtheorem*{prop*}{Proposition}
\newtheorem{cor}[theorem]{Corollary}
\newtheorem*{cor*}{Corollary}
\newtheorem{rem}[theorem]{Remark}

\newtheorem*{mt*}{Main Theorem}

\DeclareMathOperator{\Span}{Span}
\DeclareMathOperator{\id}{id}

\DeclareMathOperator{\Ker}{Ker}

\DeclareMathOperator{\im}{Im}
\DeclareMathSymbol{\Finv} {\mathord}{AMSb}{"60}

\newcommand\restrict[1]{\raisebox{-.5ex}{$|$}_{#1}}

\newcommand{\R}{\mathbb{R}}
\newcommand{\C}{\mathbb{C}}

\newcommand{\q}{\mathbb{Q}}
\newcommand{\del}{\partial}
\newcommand{\delbar}{\overline{\partial}}

\numberwithin{equation}{section}

\newcommand{\A}{\mathcal{A}}

\let\phi\varphi
\newcommand{\de}[2]{\frac{\partial #1}{\partial #2}}
\DeclareFontFamily{U}{MnSymbolC}{}
\DeclareSymbolFont{MnSyC}{U}{MnSymbolC}{m}{n}
\DeclareFontShape{U}{MnSymbolC}{m}{n}{
    <-6>  MnSymbolC5
   <6-7>  MnSymbolC6
   <7-8>  MnSymbolC7
   <8-9>  MnSymbolC8
   <9-10> MnSymbolC9
  <10-12> MnSymbolC10
  <12->   MnSymbolC12}{}
\DeclareMathSymbol{\intprod}{\mathbin}{MnSyC}{'270}

\allowdisplaybreaks

\author{Tommaso Sferruzza}
\address[Tommaso Sferruzza]{
Dipartimento di Scienze Matematiche, Fisiche e Informatiche\\
Unità di Mate\-matica e Informatica\\
Università degli studi di Parma}
\email{tommaso.sferruzza@unipr.it}

\title[Deformations of astheno-K\"ahler metrics]{Deformations of astheno-K\"ahler metrics}

\thanks{The author has been supported by GNSAGA of INdAM.}
%{53C55, 53C44, 32J15, 57T10} %16E45

\date{\today}

\begin{document}
\begin{abstract}
The property of admitting an astheno-K\"ahler metric is not stable under the action of small deformations of the complex structure of a compact complex manifold. In this paper, we prove necessary cohomological conditions for the existence of curves of astheno-K\"ahler metrics along curves of deformations starting from an initial compact complex manifold endowed with an astheno-K\"ahler metric. Furthermore, we apply our results providing obstructions to the existence of curves of astheno-K\"ahler metrics on two different families of real $8$-dimensional nilmanifolds endowed with invariant nilpotent complex structures.
\end{abstract}

\maketitle

\section{Introduction}

The study of small deformations of the complex structure of a compact complex manifold yields many interesting insights on its homotopy and complex invariants. A remarkable result by Kodaira and Spencer in \cite {KS} is the proof of the stability under small deformations of the K\"ahler condition, i.e., the property of admitting a Hermitian metric whose fundamental form is closed. As a matter of fact, the K\"ahler condition is strongly related to the topology of a compact even dimensional manifold, e.g., it forces many costraints on the Betti numbers, that is, the dimensions of the de Rham cohomology spaces, and it implies the formality according to Sullivan, that is, the algebra of differential forms is equivalent to the algebra of its de Rham cohomology (\cite{Sul75,Sul77,DGMS}). Stability under deformation has been proved also for important invariants directly related to the complex structure of a manifold, e.g., the property of satisfying the $\del\delbar$-lemma, i.e., every $d$-closed $d$-exact form is also $\del\delbar$-exact, and different versions that have been recently introduced and studied (see \cite{AU,RWZ,SW}).

Concerning the stability of special Hermitian structures, i.e., Hermitian metrics whose fundamental forms belong to the kernel of differential operators depending on the complex structure of the manifold, the property of admitting a \emph{standard metric in the sense of Gauduchon} has been proved to be stable. In fact, by the celebrated Gauduchon theorem (see \cite{Gau}), given any Hermitian metric on a complex manifolds, its fundamental form is conformal to a standard metric in the sense of Gauduchon. However, this class of metrics represents a vary special case, since for many other notions of special metrics on a $n$-dimensional complex manifold, e.g.,
\begin{itemize}
\item \emph{strong K\"ahler with torsion} metrics, i.e., metrics whose fundamental form $\omega$ satisfies $\del\delbar\omega=0$,
\item \emph{astheno-K\"ahler} metrics, i.e., metrics whose fundamental form $\omega$ satisfies $\del\delbar\omega^{n-2}=0$,
\item \emph{astheno-K\"ahler} metrics, i.e, metrics whose fundamental form $\omega$ satisfies $d\omega^{n-1}=0$,
\end{itemize}
stability does not hold in general (see, respectively, \cite{FT09,AB90}). On the other hand, sufficient conditions for the stability of balanced and SKT metrics have been proved in \cite{AU,Ma,RZ18}. Hence, it is natural to investigate whether there exist assumptions on either the initial manifold or the performed deformation of the complex structure of a compact complex manifold so that the astheno-K\"ahler condition is stable.

The notion of astheno-K\"ahler metrics has been first introduced by Jost and Yau in the study of solutions to certain non linear elliptic systems \cite{JY} and has been later used to prove rigidity theorems regarding projectively flat manifolds in \cite{LYZ}. %Such metric structures can be defined as follows. Let $(M,J)$ be a $n$-dimensional complex manifold and $(g,\omega)$ a Hermitian structure on $(M,J)$, i.e., the datum of a Hermitian metric $g$ and its associated fundamental form $\omega=g(J\cdot,\cdot)$. Then, the metric $g$ is \emph{astheno-K\"ahler} if
%\[
%\del\delbar\omega^{n-2}=0.
%\]
By definition, it is clear that on a complex manifold of complex dimension lesser than $2$ every Hermitian metric is an astheno-K\"ahler metric. In complex dimension $3$, the notion of astheno-K\"ahler coincides with the notion of strong K\"ahler with torsion and, hence, the results valid for the latter class of metrics hold also for the former, e.g., on $6$-dimensional real nilmanifold endowed with an invariant complex structure, the existence of invariant SKT (and hence, astheno-K\"ahler) metrics depends only on the complex structure (see \cite{FPS}). Nonetheless, for $n>3$, this notions are not related, i.e., there exists examples of manifolds with astheno-K\"ahler metrics which are not strong K\"ahler with torsion, and viceversa (see \cite{FT11, RosTom}).

Furthermore, whereas for SKT metrics it has been conjectured that on the same non-K\"ahler compact complex manifold there cannot exist both SKT metrics and balanced metrics with respect to the same complex structure (see \cite{FV}), in \cite{MT} it has been proved that an astheno-K\"ahler  metric is balanced if and only if the metric is also K\"ahler. Nonetheless, in \cite{FGV} the authors show the existence of a compact complex non-K\"ahler manifold which admits both a balanced and astheno-K\"ahler metric. Moreover, the existence of SKT or balanced metrics is equivalent to the existence of, respectively, $1$-pluriclosed and $(n-1)$-K\"ahler structures on a manifold (see \cite{AA}), so that their existence on a compact manifolds can be intrinsically characterized in terms of currents (see \cite{E}), analogously to the K\"ahler case (see \cite{HL}). Even though for astheno-K\"ahler metrics such a characterization does not hold, the existence of an astheno-K\"ahler metric yet implies certain restraint on a manifold, i.e., the existence of a $(n-2)$-pluriclosed form and the closedness with respect to the exterior differentential of any holomorphic $1$-form. In \cite{FT11}, the authors study the behaviour of astheno-K\"ahler metrics under complex blowup and identify a sufficient differential condition for stability; in \cite{SfeTom2} it is shown that under weaker differential conditions, stability is not assured.

In this paper, following a similar approach to \cite{PS} and \cite{Sf}, we prove necessary conditions of cohomological type for the existence of a curve of astheno-K\"ahler metrics starting from a fixed astheno K\"ahler metric on a compact complex manifold. More specifically, if $e^{\iota_{\phi}|\iota_{\overline{\phi}}}$ denotes the extension map introduced in \cite{RZ18} and recalled in section \ref{sec:prel}, and $\iota_{\psi}$ is the contraction by a $(0,1)$-vector form with values in a holomorphic vector bundle, as we recall in section \ref{sec:formulas},
%let $(M,J)$ be a compact $n$-dimensional complex manifold and let $g$ be an astheno-K\"ahler metric with fundamental form $\omega$. Let $\{M_t\}_{t\in I}$, $I=(-\epsilon,\epsilon)$, $\epsilon>0$, be a curve of deformations parametrized by a $(0,1)$-vector form $\phi(t)$, $t\in I$, with values in $T^{1,0}(M)$ and let $\{\omega_t\}_{t\in I}$ be a curve of Hermitian metrics along the curve $\{M_t\}_{t\in I}$. In local coordinates, write $\omega=\omega_{i\overline{j}}dz^i\wedge d\overline{z}^j$, and through the exponential map $e^{i_{\phi}|i_{\overline{\phi}}}$ (see Section \ref{sec-astheno-prel-def}), write $\omega_t=e^{i_{\phi}|i_{\overline{\phi}}}(\omega(t))=e^{i_{\phi}|i_{\overline{\phi}}}(\omega_{i\overline{j}}(t)dz_i\wedge d\overline{z}^j)$, so that
%\[
%\omega_t^{n-2}=e^{i_{\phi}|i_{\overline{\phi}}}(\omega^{n-2}(t))=e^{i_{\phi}|i_{\overline{\phi}}}(\omega_{i_1\overline{j}_1}(t)\dots\omega_{i_{n-2}\overline{j}_{n_2}}(t)dz^{i_1}\wedge\dots\wedge dz^{ i_{n-2}}\wedge d\overline{z}^{j_1}\wedge\dots\wedge d\overline{z}^{j_{n-2}}).
%\] 
%Setting
%\begin{equation*}
%(\omega^{n-2}(t))':=\de{}{t}(\omega_{i_1\overline{j}_1}(t)\dots\omega_{i_{n-2}\overline{j}_{n_2}}(t))dz^{i_1}\wedge\dots\wedge dz^{ i_{n-2}}\wedge d\overline{z}^{j_1}\wedge\dots\wedge d\overline{z}^{j_{n-2}},
%\end{equation*}
we have obtained the following results.
\begin{theorem*}
Let $(M,J)$ be a $n$-dimensional compact complex manifold endowed with an astheno-K\"ahler metric $g$ and associated fundamental form $\omega$. Let $\{M_t\}_{t\in I}$ be a differentiable family of compact complex manifolds with $M_0=M$ and parametrized  by $\phi(t)\in\A^{0,1}(T^{1,0}(M))$, for $t\in I:=(-\epsilon,\epsilon)$, $\epsilon>0$. Let $\{\omega_t\}_{t\in I}$ be a smooth family of Hermitian metrics along $\{M_t\}_{t\in I}$, written as
\begin{equation*}
\omega_t=e^{\iota_{\phi}|\iota_{\overline{\phi}}}\,\,(\omega(t)),
\end{equation*}
where, locally,  $\omega(t)=\omega_{ij}(t)\, dz^i\wedge d\overline{z}^j\in\mathcal{A}^{1,1}(M)$ and $\omega_0=\omega$.

If $\omega_t^{n-2}$ has local expression $e^{\iota_{\phi}|\iota_{\overline{\phi}}}(\omega^{n-2}(t))=e^{\iota_{\phi}|\iota_{\overline{\phi}}}(f_v(t)\,dz^{i_1}\wedge d\overline{z}^{j_1}\wedge\dots\wedge dz^{i_{n-2}}\wedge d\overline{z}^{j_{n-2}})$, set
\[(\omega^{n-2}(t))':=\de{}{t}(f_v(t))\, dz^{i_1}\wedge d\overline{z}^{j_1}\wedge\dots dz^{i_{n-2}}\wedge d\overline{z}^{j_{n-2}}\in\A^{n-2,n-2}(M).
\]
Then, if every metric $\omega_t$ is astheno-K\"ahler, for $t\in I$, it must hold that
\begin{equation*}
2i\mathfrak{Im}(\del\circ \iota_{\phi'(0)}\circ\del) (\omega^{n-2})=\del\delbar (\omega^{n-2}(0))'.
\end{equation*}
\end{theorem*}
As a direct consequence, we immediately have the following corollary.
\begin{cor*}
Let $(M,J)$ be a compact Hermitian manifold endowed with an asteno-K\"ahler metric $g$ and associated fundamental form $\omega$. If there exists a smooth family of astheno-K\"ahler metrics which coincides with $\omega$ in $t=0$, along the family of deformations $\{M_t\}_t$ with $M_0=M$ and parametrized by the $(0,1)$-vector form $\phi(t)$ on $M$, then the following equation must hold
\begin{equation*}
\left[(\del\circ \iota_{\phi'(0)}\circ\del) (\omega^{n-2})\right]_{H_{BC}^{n-1,n-1}(M)}=0.
\end{equation*}
\end{cor*}
We note that the proof of the main theorem makes use of the formulas proved in \cite{RZ18} for the differential operators $\del_t$ and $\delbar_t$ acting on $(p,q)$-forms on each $M_t$, $t\in I$, combined with the Taylor series expansion centered in $t=0$.

As an application of our results, we characterize the obstructions to the existence of curves of astheno-K\"ahler metrics along certain curves of deformations, starting from two families of $4$-dimensional nilmanifolds endowed with invariant nilpotent complex structures. Thanks to general results (see e.g., \cite{ConFin,AngKas,Ang}), nilmanifolds with invariant complex structures (see, e.g., \cite{CFGU}) are a natural source of examples of compact complex non-K\"ahler manifolds on which both cohomological and metric properties can be explicitly verified at the level of invariant tensors. %(see, e.g.,  \cite{ConFin,AngKas,Ang,FT11,RosTom, SfeTar, SfeTom1, SfeTom2, PS, Sf, TarTom}).
%Furthermore, in Lemma \ref{lem:tecnico_astheno_def} we exploit the complex structure of the first family of manifolds to outline an approach to find obstructions by means of Corollary \ref{cor:def_curve_astheno}; such an approach is then extended to families of nilmanifolds with similar complex structure in any dimension. 

The paper is organized as follows. In section \ref{sec:prel}, we fix the notations that will be used throughout the paper. In section \ref{sec:formulas}, we briefly recall the necessary tools of deformation theory, among which the extension map and the formulas for the differential operators $\del_t$ and $\delbar_t$ acting on differential forms on a family of deformations. In section \ref{sec:main}, we will state and prove the main theorem. In section \ref{sec:nilm_nilp_cplx_struct}, we recall the main facts about the geometry of nilmanifolds and special complex structures and in section \ref{sec:appl} we characterize the obstructions yielded by Theorem \ref{thm:def_curve_astheno} and Corollary \ref{cor:def_curve_astheno} starting from two families of $4$-dimensional nilmanifolds with invariant nilpotent complex structure.
\vspace{0.5cm}

\emph{Acknowledgements.} The author would like to kindly thank Adriano Tomassini for many useful suggestions and comments.

\section{Notation}\label{sec:prel}

Let $(M,J)$ be a compact complex manifold of complex dimension $n$, where $M$ is a $2n$-dimensional compact differentiable manifold and $J$ is a {\em integrable almost complex structure} on $M$, i.e., an endomorphism of $TM$ such that $J^2=-\id_{TM}$ and the \emph{Nijenhuis tensor} $N_J$ identically vanishes, that is,
\begin{equation*}
N_J(X,Y)=[X,Y]-J[JX,Y]-J[X,JY]-[JX,JY]=0
\end{equation*}
for every $X,Y\in TM$.

Once the complex structure $J$ is extended to the complexified version of the bundles $TM$, we obtain the decomposition
\begin{equation*}
T_\C M=T^{1,0}M\oplus T^{0,1}M
\end{equation*}
in terms of the $\pm i$-eigenspaces of $J$. A similar decomposition holds when we consider the induced action of $J$ on complex forms, i.e.,
\begin{equation*}
T_\C^*M=(T^{1,0}M)^*\oplus (T^{0,1}M)^*
\end{equation*}
so that on the exterior powers $\bigwedge_{\C}^{k}M:=\bigwedge^k(T_\C^*M)$ it holds
\begin{equation*}
\textstyle\bigwedge_\C^kM=\displaystyle\bigoplus_{p+q=k}\textstyle\bigwedge^{p,q}M,
\end{equation*}
where $\bigwedge^{p,q}M:=\bigwedge^{p}(T^{1,0}M)^*\otimes \bigwedge^{q}(T^{0,1}M)^*$ is the bundle of $(p,q)$-forms on $M$. Let $\mathcal{A}_{\C}^k(M)$ and $\mathcal{A}^{p,q}(M)$ denote the global sections of the bundles of, respectively, complex $k$-forms and $(p,q)$-forms on $(M,J)$. Then, since $J$ is integrable, the differential $d$ acts on $(p,q)$-forms on $M$ as
\begin{equation*}
d\colon \mathcal{A}^{p,q}(M)\rightarrow \mathcal{A}^{p+1,q}(M)\oplus \mathcal{A}^{p,q+1}(M),
\end{equation*}
so that $d=\del+\delbar$, where we have set $\del:=\pi^{p+1,q}$ and $\delbar:=\pi^{p,q+1}$. Clearly, $\del^2=\delbar^2=0$ and $\del\delbar=-\delbar\del$.

Let $g$ be a Hermitian metric on $(M,J)$, i.e., a Riemannian metric on $M$ such that $g(JX,JY)=g(X,Y)$ for every $X,Y\in TM$. Let $\omega(X,Y):=g(JX,Y)$ for every $X,Y\in TM$, be the fundamental form associated to $g$. The extension of $\omega$ to a form  on $\mathcal{A}_\C^{2}M$ is a positive $(1,1)$-form on $M$. A \emph{astheno-K\"ahler metric} on a $n$-dimensional complex manifold $(M,J)$ is a Hermitian metric $g$ such that its fundamental form satisfies
\begin{equation*}
\del\delbar\omega^{n-2}=0.
\end{equation*}
A \emph{strong K\"ahler with torsion metric} on $(M,J)$ is a Hermitian metric $g$ such that its fundamental form satisfies
\begin{equation*}
\del\delbar\omega=0.
\end{equation*}

The \emph{Bott-Chern cohomology} of $(M,J)$ is given by the spaces
\begin{equation*}
H_{BC}^{p,q}(M):=\frac{\Ker\left(d\colon \mathcal{A}^{p,q}(M)\rightarrow\mathcal{A}^{p+1,q}(M)\oplus \mathcal{A}^{p,q+1}(M)\right)}{\im\left(\del\delbar\colon \mathcal{A}^{p-1,q-1}(M)\rightarrow\mathcal{A}^{p,q}(M)\right)},
\end{equation*}
whereas we say that a $(p,q)$-form $\alpha$ on $M$ is said \emph{Bott-Chern harmonic} if $\Delta_{BC}\alpha=0$, where the \emph{Bott-Chern Laplacian}
\begin{equation*}
\Delta_{BC}=\del\delbar\delbar^*\del^*+
\delbar^*\del^*\del\delbar+\delbar^*\del\del^*\delbar+
\del^*\delbar\delbar^*\del+\delbar^*\delbar+\del^*\del
\end{equation*} 
is the fourth-order self-adjoint elliptic operator. Since $M$ is assumed to be compact, $\alpha$ is Bott-Chern harmonic if, and only if, $d\alpha=0$ and $\del\delbar\ast\alpha=0$, where $\ast$ is the $\C$-antilinear Hodge  $\ast$-operator with respect to $g$. Let us denote $\mathcal{H}_{\Delta_{BC}}^{p,q}(M)$ the space of Bott-Chern harmonic $(p,q)$-forms. Hodge theory adapted to Bott-Chern cohomology by Schweitzer in \cite{Sw} implies that there exist the following isomorphisms of vector spaces
\begin{equation*}
\mathcal{H}_{\Delta_{BC}}^{p,q}(M)\xrightarrow{\sim} H_{BC}^{p,q}(M).
\end{equation*}

Let $\pi\colon E\rightarrow M$ be a holomorphic vector bundle of rank $r$ on $M$. Then, we set $\bigwedge^{p,q}(M,E):=\bigwedge^{p,q}M\otimes E$ for the bundle of \emph{$(p,q)$-forms with values in $E$}, and $\mathcal{A}^{p,q}(M,E)$ will denote the its global sections.

We can define a $\delbar_E$ operator on $\mathcal{A}^{p,q}(M,E)$ in the following way. Let $\phi=\sum_i\eta^i\otimes s_i\in\mathcal{A}^{p,q}(M,E)$, where each $\eta^i\in\mathcal{A}^{p,q}(M)$ and $(s_1,\dots,s_r)$ is the expression of $s$ in a local triviallization on $E$. Then
\begin{equation}
\delbar_E\phi:= \sum_{i}\delbar \eta^i\otimes s_i.
\end{equation} 

Throughout this paper, we will consider essentially as $E$ the holomorphic tangent bundle $T^{1,0}M$ and its conjugate $T^{0,1}M$ and we will refer the sections of $\mathcal{A}^{0,q}(M,T^{1,0}M)$ as \emph{$(0,q)$-vector forms} on $M$; we will omit the term ``$M$" in the notations when understood.

We also can define contraction by $(0,r)$-vector forms on any holomorphic vector bundle $E$ over $M$. Let $\xi=\alpha\otimes s\in\mathcal{A}^{p,q}(E)$ and $\phi=\overline{\eta}\otimes Z\in\mathcal{A}^{0,r}(T^{1,0}M)$. Then the \emph{contraction by $\phi$} is defined as
\begin{gather*}
\iota_\phi\colon \mathcal{A}^{p,q}(E)\rightarrow \mathcal{A}^{p-1,q+r}(E)\\
\xi \mapsto \iota_{\phi}(\xi):=\overline{\eta}\wedge \iota_{Z}(\alpha)\otimes s.
\end{gather*} 
By linearity, we can extend such definition to respectively $\mathcal{A}^{p,q}(E)$ and $\mathcal{A}^{0,s}(T^{1,0}M)$. If one considers the conjugate $\overline{\phi}$ of a $(0,q)$-vector form on $M$, i.e., $\overline{\phi}\in\mathcal{A}^{q,0}(T^{0,1}M)$, a contraction can be analogously defined by setting $\iota_{\overline{\phi}}(\xi):=\eta\wedge \iota_{\overline{Z}}(\alpha)\wedge s$, for every $\xi\in\mathcal{A}^{p,q}(E)$. We will interchangeably use the notations $\phi\intprod$ for $\iota_{\phi}$ and $\overline{\phi}\intprod$ for $\iota_{\overline{\phi}}$.

\section{Review of deformation theory and the operators $\del_t$ and $\delbar_t$ along curves of deformations}\label{sec:formulas}
In this section we recall the formulas fo the action of the differential operators $\del_t$ and $\delbar_t$ on each element of a differentiable family of deformations $\{M_t\}$ of a compact complex manifold $(M,J)$; we follow the approach in \cite{RZ18}, in which such operators are expressed in terms of the operators $\del$ and $\delbar$ on $(M,J)$ and on a $(0,1)$-vector form which parametrizes the differential family $\{M_t\}_{t}$.
 
Let $B$ be a differential manifold of real dimension $m$.
\begin{definition}[\cite{KM}]
A \emph{differentiable family of compact complex manifold} is a differentiable manifold $\mathcal{M}$ and a differentiable proper map $\pi\colon \mathcal{M}\rightarrow B$ such that
\begin{itemize}
\item for every $t\in B$, the fiber $M_t:=\pi^{-1}(t)$ has a structure of complex manifold;
\item the rank of the Jacobian of $\pi$ is constantly equal to $m$, at every point $p\in\mathcal{M}$.
\end{itemize}
\end{definition}
If $(M,J)$ is a compact complex manifold, then a differentiable family of compact complex manifolds is a \emph{differential family of deformations} of $(M,J)$ if there exists $t_0\in B$ such that $M_{t_0}=(M,J)$. We refer to this fixed element of the family as the \emph{central fiber}. Without loss of generality, we can assume that the space of parameters $B$ to be a polydisc in $\R^m$, i.e., $$B:=\{(t_1,\dots,t_m)\in\R^m: |t_j|<\epsilon,j\in\{1,\dots,m\}\},$$ for $\epsilon>0$, and $t_0$ will always be $0\in\R^m$. 

Viceversa, starting from a compact complex manifold $(M,J)$, in \cite{KM} it is shown that one can construct families of deformations $\{M_t\}_{t\in B}$ of $(M,J)$ by setting $$M_t:=(M,J_t), \qquad\text{for every}\quad t\in B,$$ where $J_t$ is an integrable complex structure on the differentiable manifold $M$ parametrized by a $(0,1)$-vector form $\phi(t)$ on the central fiber $(M,J)$, i.e, $\phi(t)\in\mathcal{A}^{0,1}(T^{1,0}M)$. From $M_0=(M,J)$, it follows that $\phi(0)=0$. We remark that in order for each $J_t$ to define an integrable complex structure on $M$, the $(0,1)$-vector form $\phi(t)$ must satisfy the \emph{Maurer-Cartan equation}, i.e.,
\begin{equation}\label{eq:MC_phi}
\delbar\phi(t)-\frac{1}{2}[\phi(t),\phi(t)]=0, \qquad\text{for every}\quad t\in B.
\end{equation}
For the sake of semplicity, we will set the dimension of the parameters space $B$ to be $m=1$, i.e., $B=I=(-\epsilon,\epsilon)$, for $\epsilon>0$.

Let then $(M,J)$ be a compact $n$-dimensional complex manifold and let $\phi(t)\in\mathcal{A}^{0,1}(T^{1,0}(M))$, $t\in I$, parametrize a differentiable family of deformations $\{M_t\}_{t\in I}$ of $(M,J)$. We will refer to differential families over an interval $I$ as \emph{curves of deformations}.

We now recall a map which links the $(p,q)$-forms on the central fiber $(M,J)$ and on every element of the family $M_t=(M,J_t)$. First, let us then define the maps
\begin{equation}
e^{\iota_{\phi(t)}}:=\sum_{j=1}^{\infty}\iota_{\phi(t)}^k, \qquad e^{\iota_{\overline{\phi(t)}}}:=\sum_{j=1}^{\infty}\iota_{\overline{\phi(t)}}^k,
\end{equation}
where $i_\psi^k$ is the contraction by the vector form $\psi$ repeated $k$ times. Note that since $M$ is compact, each summation is finite.

If $\alpha=\alpha_{i_1\dots i_p j_1\dots j_q }dz^{i_1}\wedge\dots\wedge dz^{i_p}\wedge d\overline{z}^{j_1}\wedge\dots\wedge  d\overline{z}^{j_q}$ is the local expression of a $(p,q)$-form $\alpha$ on $(M,J)$, then the \emph{extension map} $e^{\iota_{\phi(t)}|\iota_{\overline{\phi}(t)}}$ is defined as
\begin{equation*}
e^{\iota_{\phi(t)}|\iota_{\overline{\phi(t)}}}(\alpha)=\alpha_{i_1\dots i_p j_1\dots j_q }e^{\iota_{\phi(t)}}(dz^{i_1}\wedge dz^{i_p})\wedge e^{\iota_{\overline{\phi}(t)}}(d\overline{z}^{j_1}\wedge d\overline{z}^{j_q}).
\end{equation*}
Note that such a map defines a global object, since $\phi(t)$  is a global $(0,1)$-vector form on $(M,J)$. At the level of each space of $(p,q)$-forms, we have the following result, see \cite[Lemma 2.9]{RZ18}
\begin{lemma}\label{lemma:rao-zhao}
For a fixed $t\in I$, the exponential map defines the following isomorphisms
\begin{gather*}
e^{\iota_{\phi(t)}|\iota_{\overline{\phi(t)}}}\colon \mathcal{A}^{p,q}(M)\rightarrow\mathcal{A}^{p,q}(M_t).
\end{gather*}
\end{lemma}

Hence, once we fix $t\in I$, we can see any $(p,q)$-form $\alpha_t$ on $M_t$ as the image $e^{\iota_{\phi(t)}|\iota_{\overline{\phi}(t)}}(\alpha)$ of a certain $\alpha\in\mathcal{A}^{p,q}(M)$.

Exploiting Lemma~\ref{lemma:rao-zhao}, the Maurer-Cartan equation \eqref{eq:MC_phi} for the integrability of $J_t$ can be restated in the following versatile terms. Let $\{\eta^1,\dots,\eta^n\}$ be a global frame of $(1,0)$-forms on $(M,J)$, and $\{\eta_{t}^j:=e^{\iota_{\phi(t)}|\iota_{\overline{\phi}(t)}}(\eta^j)\}_{j=1}^n$ the corresponding frame of $(1,0)$-forms on $M_t$. Then $J_t$ defines an integrable complex structure on $M$ if
\begin{equation*}
(d\eta_t^j)^{0,2}=0,\qquad \text{for every}\, j\in\{1,\dots,n\},
\end{equation*}
where $(d\eta_t^j)^{0,2}$ denotes the $(0,2)$-component of $d\eta_t^j$ with respect to the decomposition given by
\begin{equation*}
\mathcal{A}_{\C}^k(M)=\bigoplus_{p+q=k}\mathcal{A}^{p,q}(M_t).
\end{equation*}
Note that this condition is equivalent to \eqref{eq:MC_phi}.

We end this section  by recalling the formulas for the operators $\del_t$ and $\delbar_t$ acting on $(p,q)$ forms on each $M_t$. We will usually omit the dependence on $t$ of the $(0,1)$-vector form $\phi(t)$.

Let $\{M_t\}_{t\in I}$ be a curve of deformations of $(M,J)$ such that each $M_t=(M,J_t)$ and each complex structure $J_t$ is parametrized by a $(0,1)$-vector form $\phi(t)$, $t\in I$, on $(M,J)$. Then, from the expression of $d$ at that level of $(p,q)$-forms on each $M_t$, i.e., 
\begin{equation*}
d\colon \mathcal{A}^{p,q}(M_t)\rightarrow \mathcal{A}^{p+1,q}(M_t)\oplus \mathcal{A}^{p,q+1}(M_t)
\end{equation*}
the operators $\del_t$ and $\delbar_t$ are, by definition, $\del_t:=\pi_t^{p+1,q}\circ d$ and $\delbar_t:=\pi_t^{p,q+1}\circ d$.
In \cite{RZ18}, the authors provides the formulas for the action of such operators on both functions and $(p,q)$-forms. Let then $f\colon M\rightarrow\C$ be a smooth function on $M$. Then
\begin{gather*}
\del_t f=e^{i_\phi}\left((I-\phi\overline{\phi})^{-1}\intprod(\del-\overline{\phi}\intprod\delbar)f\right)\\
\delbar_t f=e^{\iota_{\overline{\phi}}}\left((I-\overline{\phi}\phi)^{-1}\intprod(\delbar-\phi\intprod\del)f\right),
\end{gather*}
where we use the notations $\phi\overline{\phi}=\overline{\phi}\intprod\phi$, $\overline{\phi}\intprod\phi=\phi\intprod\overline{\phi}$, and $I$ is the identity map (see \cite[Formula 2.13]{RZ18}). In particular, a map $f\colon M\rightarrow \C$ is holomorphic with respect to the complex structure $J_t$ if, and only if, $$\del f -(\overline{\phi(t)}\intprod\delbar)f=0.$$

We recall that the \emph{simultaneous contraction} $\Finv$ by a $(0,1)$-vector form $\phi$ (or any operator) acts on a $(p,q)$ form $\alpha$ as
\begin{equation*}
\phi\Finv\alpha:=\alpha_{i_1\dots i_{p}j_1\dots j_{q}}\phi\intprod dz^{i_1}\wedge \dots \wedge \phi\intprod dz^{i_p}\wedge \phi \intprod d\overline{z}^{j_1}\wedge\dots\wedge \phi\intprod d\overline{z}^{j_q},
\end{equation*}
where $\alpha_{i_1\dots i_{p}j_1\dots j_{q}} dz^{i_1}\wedge\dots\wedge dz^{i_p}\wedge d\overline{z}^{j_1}\wedge\dots \wedge d\overline{z}^{j_q}$ is the local expression for $\alpha=$. Such a contraction is well defined and can be used to rewrite the exponential map as
\begin{equation}\label{eq:simcont_extension}
e^{\iota_{\phi}|\iota_{\overline{\phi}}}=(I+\phi+\overline{\phi})\Finv.
\end{equation}
Then, from \cite[Prop 2.13]{RZ18}, we obtain that the action of $\del_t$ and $\delbar_t$ on any $e^{\iota_{\phi}|\iota_{\overline{\phi}}}\alpha\in\mathcal{A}^{p,q}(M_t)$, for $\alpha\in\mathcal{A}^{p,q}(M)$, is defined as
\begin{gather}
\del_t(e^{\iota_{\phi}|\iota_{\overline{\phi}}}\alpha)=e^{\iota_{\phi}|\iota_{\overline{\phi}}}\left((I-\phi\overline{\phi})^{-1}\Finv([\delbar,\iota_{\overline{\phi}}]+\del)(I-\phi\overline{\phi})\Finv\alpha\right)\label{delt}\\
\delbar_t(e^{\iota_{\phi}|\iota_{\overline{\phi}}}\alpha)=e^{\iota_{\phi}|\iota_{\overline{\phi}}}\left((I-\overline{\phi}\phi)^{-1}\Finv([\del,\iota_{\phi}]+\delbar)(I-\overline{\phi}\phi)\Finv\alpha\right).\label{delbart}
\end{gather}
\section{Deformations of Astheno-K\"ahler metrics}\label{sec:main}

Let $(M,J)$ be a compact complex manifold of complex dimension $n$ endowed with an astheno-K\"ahler metric $g$, i.e., its associated fundamental form $\omega$ satisfies $\del\delbar\omega^{n-2}=0$, and let $\{M_t\}_{t\in I}$ be a curve of deformations of $(M,J)$, with $\{M_t\}_{t\in I}$ parametrized by a $(0,1)$-vector form $\phi(t)$ on $M$. Let also $\{\omega_t\}_{t\in I}$ be a smooth family of Hermitian metrics along $\{M_t\}_{t\in I}$, such that $\omega_0=\omega$ and suppose that $g_t$ is balanced on $M_t$, for every $t\in I$, i.e, 
\begin{equation}\label{eq:astheno_omega_t=0}
\del_t\delbar_t\omega_t^{n-2}=0, \quad \forall t\in I.
\end{equation}
By Lemma \ref{lemma:rao-zhao}, we can write each $\omega_t$ as $e^{\iota_{\phi}|\iota_{\overline{\phi}}}\omega(t)$, where for every $t\in I$, $\omega(t)=\omega_{ij}(t)dz^i\wedge d\overline{z}^j\in\mathcal{A}^{1,1}(M)$ locally, and also
\begin{equation}
\omega_t^{n-2}=e^{\iota_{\phi}|\iota_{\overline{\phi}}}(\omega^{n-2}(t))=e^{\iota_{\phi}|\iota_{\overline{\phi}}}\left(f_v(t)dz^{i_1}\wedge d\overline{z}^{j_1}\wedge\dots\wedge dz^{i_{n-2}}\wedge d\overline{z}^{j_{n-2}}\right),
\end{equation}
with $f_v(t):=\omega_{i_{1} j_{1}}(t)\dots\omega_{i_{n-2}j_{n-2}}(t)$, where we have set $v=(i_1,j_1,\dots,i_{n-2},j_{n-2})$, $i_k,j_k\in\{1,\dots,n\}$,  and $k\in\{1,\dots, n-2\}$.

We can then apply formulas \eqref{delt} and \eqref{delbart} to \eqref{eq:astheno_omega_t=0}, and by making use of Taylor series expansion and differentiating with respect to $t$ in $t=0$, we are able to prove the main theorem.

\begin{theorem}\label{thm:def_curve_astheno}
Let $(M,J)$ be a $n$-dimensional compact complex manifold endowed with an astheno-K\"ahler metric $g$ and associated fundamental form $\omega$. Let $\{M_t\}_{t\in I}$ be a differentiable family of compact complex manifolds with $M_0=M$ and parametrized  by $\phi(t)\in\A^{0,1}(T^{1,0}(M))$, for $t\in I:=(-\epsilon,\epsilon)$, $\epsilon>0$. Let $\{\omega_t\}_{t\in I}$ be a smooth family of Hermitian metrics along $\{M_t\}_{t\in I}$, written as
\begin{equation*}
\omega_t=e^{\iota_{\phi}|\iota_{\overline{\phi}}}\,\,(\omega(t)),
\end{equation*}
where, locally,  $\omega(t)=\omega_{ij}(t)\, dz^i\wedge d\overline{z}^j\in\mathcal{A}^{1,1}(M)$ and $\omega_0=\omega$.

If $\omega_t^{n-2}$ has local expression $e^{\iota_{\phi}|\iota_{\overline{\phi}}}(\omega^{n-2}(t))=e^{\iota_{\phi}|\iota_{\overline{\phi}}}(f_v(t)\,dz^{i_1}\wedge d\overline{z}^{j_1}\wedge\dots\wedge dz^{i_{n-2}}\wedge d\overline{z}^{j_{n-2}})$, set
\[(\omega^{n-2}(t))':=\de{}{t}(f_v(t))\, dz^{i_1}\wedge d\overline{z}^{j_1}\wedge\dots dz^{i_{n-2}}\wedge d\overline{z}^{j_{n-2}}\in\A^{n-2,n-2}(M).
\]
Then, if every metric $\omega_t$ is astheno-K\"ahler, for $t\in I$, it must hold that
\begin{equation}\label{eq_thm_def_curve_astheno}
2i\mathfrak{Im}(\del\circ \iota_{\phi'(0)}\circ\del) (\omega^{n-2})=\del\delbar (\omega^{n-2}(0))'.
\end{equation}
\end{theorem}
As a direct consequence, we immediately have the following corollary.
\begin{cor}\label{cor:def_curve_astheno}
Let $(M,J)$ be a compact Hermitian manifold endowed with an asteno-K\"ahler metric $g$ and associated fundamental form $\omega$. If there exists a smooth family of astheno-K\"ahler metrics which coincides with $\omega$ in $t=0$, along the family of deformations $\{M_t\}_t$ with $M_0=M$ and parametrized by the $(0,1)$-vector form $\phi(t)$ on $M$, then the following equation must hold
\begin{equation}\label{eq_cor_def_curve_astheno}
\left[(\del\circ \iota_{\phi'(0)}\circ\del) (\omega^{n-2})\right]_{H_{BC}^{n-1,n-1}(M)}=0.
\end{equation}
\end{cor}
\begin{proof}[Proof of Theorem \ref{thm:def_curve_astheno}] The metrics $\omega_t$ are astheno-K\"ahler for every $t\in I$, i.e., $\del_t\delbar_t\omega_t^{n-2}=0$. This implies
\begin{equation}\label{eq:ddt(del_tdelbar_t)omega_n-2}
\frac{\del}{\del t}(\del_t\delbar_t\omega_t^{n-2})=0.
\end{equation}
Let us then compute the right hand side of \eqref{eq:ddt(del_tdelbar_t)omega_n-2} through formulas \eqref{delt} and \eqref{delbart} for the operators $\del_t$ and $\delbar_t$. By the extension map we have that
\begin{equation*}
\del_t\delbar_t(\omega_t^{n-2})=\del_t\delbar_t(e^{i_\phi|\iota_{\overline{\phi}}}(\omega^{n-2}(t))),
\end{equation*}
and then, by \eqref{delt} and \eqref{delbart}, we have
\begin{align*}
&\del_t\delbar_t(e^{\iota_{\phi}|\iota_{\overline{\phi}}}(\omega^{n-2}(t))=\del_t(e^{\iota_{\phi}|\iota_{\overline{\phi}}}\left((I-\overline{\phi}\phi)^{-1}\Finv([\del,\iota_{\phi}]+\delbar)(I-\overline{\phi\phi})\Finv\,\omega^{n-2}(t)\right))\\
&= e^{\iota_{\phi}|\iota_{\overline{\phi}}}\Biggl(\Bigl((I-\phi\overline{\phi})^{-1}\Finv([\delbar,\iota_{\overline{\phi}}]+\del)(I-\phi\overline{\phi})\Bigr)\Finv\Bigl((I-\overline{\phi}\phi)^{-1}\Finv([\del,\iota_{\phi}]+\delbar)(I-\overline{\phi}\phi)\Finv\,\omega^{n-2}(t)\Bigr) \Biggr).
\end{align*}
Now, we expand in Taylor series centered in $t=0$ the terms
\[
\phi(t)=t\phi'(0)+o(t), \quad \omega^{n-2}(t)=\omega^{n-2}(0)+t\omega^{n-2}(0)'+o(t)
\]
so that
\begin{equation*}
(I-\phi\overline{\phi})^{-1}=(I-\overline{\phi}\phi)^{-1}=(I-\phi\overline{\phi})=(I-\overline{\phi}\phi)=I+o(t)
\end{equation*}
Combining with \eqref{eq:simcont_extension} and usig the notation $\intprod$, we obtain that
\begin{align*}
&\del_t\delbar_t\omega_t^{n-2}=\left(I+t\phi'(0)+\overline{t\phi}'(0)\right)\Finv\,\left([\delbar,\overline{t\phi'(0)}\intprod]+\del\right)\left([\del,t\phi'(0)\intprod]+\delbar\right)\left(\omega^{n-2}(0)+t(\omega^{n-2}(0))'\right)+o(t)\\
&=\left(I+t\phi'(0)+\overline{t\phi}'(0)\right)\Finv\,\left([\delbar,\overline{t\phi'(0)}\intprod]+\del\right)\left([\del,t\phi'(0)\intprod]\omega^{n-2}(0)+\delbar\omega^{n-2}(0)+t\delbar(\omega^{n-2}(0))'\right)+o(t)\\
&=\left(I+t\phi'(0)+\overline{t\phi}'(0)\right)\Finv\,\left(-t\del(\phi'(0)\intprod\omega^{n-2}(0))+t\delbar(\overline{\phi'(0)}\intprod\delbar\omega^{n-2}(0))+t\del\delbar(\omega^{n-2}(0))'\right)+o(t)\\
&= -t\del(\phi'(0)\intprod\del\omega^{n-2}(0))+t\delbar(\overline{\phi'(0)}\intprod\delbar\omega^{n-2}(0))+t\del\delbar(\omega^{n-2}(0))'+o(t).
\end{align*}
Now, since $\del_t\delbar_t\omega_t^{n-2}=0$, for every $t\in I$, also, $\frac{\del}{\del t}\restrict{t=0}(\del_t\delbar_t\omega_t^{n-2})=0$, hence
\begin{equation*}
-\del(\phi'(0)\intprod \del\omega^{n-2}(0))+\delbar(\overline{\phi'(0)}\intprod\delbar\omega^{n-2}(0))+\del\delbar(\omega^{n-2}(0))'=0,
\end{equation*}
which is equivalent to 
\begin{equation*}
-(\del\circ \iota_{\phi'(0)}\circ\del) \omega^{n-2}+(\delbar\circ \iota_{\overline{\phi'(0)}}\circ\delbar) \omega^{n-2}+\del\delbar(\omega^{n-2}(0))'=0,
\end{equation*}
hence, concluding the proof.
\end{proof}

\section{Nilmanifolds with nilpotent complex structure}\label{sec:nilm_nilp_cplx_struct}
In this section, we briefly recall the basic notions on the geometry of nilmanifolds endowed with invariant complex structure.

A \emph{nilmanifold} $M$ is the datum of a quotient $M=\Gamma\backslash G$ of a simply connected connected Lie group $G$ by a discrete uniform subgroup $\Gamma$, such that the Lie algebra $\mathfrak{g}$ associated to the Lie group $G$ is nilpotent, i.e., the series $\{\mathfrak{g}^{(k)}\}$ ends in $\{0\}$, where
\begin{equation*}
\mathfrak{g}^{(0)}:=\mathfrak{g}, \qquad \mathfrak{g}^{(1)}:=[\mathfrak{g},\mathfrak{g}], \quad \mathfrak{g}^{(j)}:=[\mathfrak{g},\mathfrak{g}^{(j-1)}].
\end{equation*} 
An \emph{invariant complex structure} $J$ on a nilmanifold $M$ is a left-invariant integrable almost complex structure on $G$, which descendens to the quotient $M$. Such complex structures on $M$ can be characterized by a set of linearly independent left-invariant forms $\{\eta^j\}_j$ on the Lie group $G$ such that
$$
\mathfrak{g}_\C^*=\Span_\C\langle\eta^j\rangle_j\otimes \Span_\C\langle\overline{\eta}^k\rangle_k,
$$
where $\mathfrak{g}_\C^*$ is the space of left-invariant complex forms on $G$.
By left-invariance, the coframe $\{\eta^j\}_j$ then descends to the quotient $M$.

Therefore, one can construct nilmanifolds by setting
\[
\mathfrak{g}_{\C}^{*}:=\Span_{\C}\langle\eta^{j}\rangle_{j}\oplus \Span_{\C}\langle\overline{\eta}^j\rangle_{j}
\]
where $\{\eta^j\}_j\subset \mathfrak{g}_{\C}^*$ a set of linearly independent complex covectors with structure equations
\[
d\eta^{j}=\sum_{i<k}A_{ik}^j\eta^{i}\wedge\eta^k+\sum_{i,k}B_{i\overline{k}}^j\eta^{i}\wedge\overline{\eta}^k,
\]
satisfying
\begin{itemize}
\item $d^2\equiv 0$;
\item the \emph{constant structure} $A_{ik}^j$, $B_{i\overline{k}}^j\in\q[i]$;
\item if $\{Z_j\}_j$ is the dual base of $\{\eta^j\}_j$, through the relations
$$
[Z_j,Z_k]=A_{jk}^iZ_i,\quad [Z_j,\overline{Z}_k]=(B_{j\overline{k}}^iZ_i-B_{k\overline{j}}^i\overline{Z}_i),
$$
it holds that $\mathfrak{g}^{(k_0)}=\{0\}$, for some $k_0$.
\end{itemize}
Indeed, under these hypotesis, the Lie algebra $\mathfrak{g}$ underlying $\mathfrak{g}_{\C}^*$ is nilpotent, and its corresponding Lie group $G$ admits a discrete uniform subgroup $\Gamma$ by Malcev theorem, so that $(M=\Gamma\backslash G,J)$ is a nilmanifold endowed with an invariant complex structure $J$ and
\[
\mathfrak{g}_{\C}=\mathfrak{g}^{1,0}\oplus \mathfrak{g}^{0,1},\qquad  \mathfrak{g}_{\C}^*=\mathfrak{g}^{*(1,0)}\oplus\mathfrak{g}^{*(0,1)}
\]
where
\[
\mathfrak{g}^{1,0}=\Span_{\C}\langle Z^j\rangle_j, \quad \mathfrak{g}^{0,1}=\Span_{\C}\langle\overline{Z}^j\rangle_j, \quad \mathfrak{g}^{*(1,0)}=\Span_{\C}\langle \eta^j\rangle_j, \quad \mathfrak{g}^{*(0,1)}=\Span_{\C}\langle\overline{\eta}^j\rangle_j.
\]
In particular, we say that the complex structure $J$ on $M$ is
\begin{itemize}
\item \emph{holomorphically parallelizable} if
\[
d\mathfrak{g}^{*(1,0)}\subset\mathfrak{g}^{*(2,0)}
\]
\item \emph{nilpotent} if there exists a coframe of left-invariant $(1,0)$-forms $\{\eta^j\}$ on $G$ such that, for every $j$,
\[
d\eta^j=\sum_{i<k<j}A_{ij}\eta^i\wedge\eta^k +\sum_{i,k<j}B_{i\overline{k}}\eta^i\wedge\overline{\eta}^k, \qquad A_{ik}^j,B_{i\overline{k}}^j\in\C,
\]
\item \emph{abelian} if $$d\mathfrak{g}^{*(1,0)}\subset \mathfrak{g}^{*(1,1)}.$$
\end{itemize}
\section{Applications}\label{sec:appl}
We now apply Corollary \ref{cor:def_curve_astheno}, providing examples of obstructions to the existence of curve of astheno-K\"ahler metrisc on two families of $4$-dimensional nilmanifolds.
\subsection{Example 1}\label{example1}
Let $(M=\Gamma\backslash G,J)$ be the $4$-dimensional nilmanifold with left-invariant complex structure $J$ induced by the covectors $\{\eta^1,\eta^2,\eta^3,\eta^4\}$ on the complexified dual $\mathfrak{g}_{\C}^*$ of the Lie algebra $\mathfrak{g}$ of $G$, which satisfy structure equations
\begin{align}\label{eq:struct_eq_def_astheno_0}
\begin{cases}
\,\,d\eta^i&=0, \quad i\in \{1,2,3\},\\
\,\,d\eta^4&=a_1\eta^{12}+a_2\eta^{13}+a_3\eta^{1\overline{1}}+a_4\eta^{1\overline{2}}+a_5\eta^{1\overline{3}}\\
\,\,&+a_6\eta^{23}+
+a_7\eta^{2\overline{1}}+a_8\eta^{2\overline{2}}+a_9\eta^{2\overline{3}}\\
\,\,&+a_{10}\eta^{3\overline{1}}+a_{11}\eta^{3\overline{2}}+a_{12}\eta^{3\overline{3}},
\end{cases}
\end{align}
with $a_j\in\q[i]$, for every $j\in\{1,2,\dots,12\}$. Note that $J$ is a nilpotent complex structure on $M$.

If $\omega=\frac{i}{2}\sum_{j=1}^4\eta^{j\overline{j}}$ is the fundamental form associated to the diagonal metric $g$ on $M$, then $\omega$ is astheno-K\"ahler i.e., 
\begin{equation*}
\del\delbar\omega^{2}=0
\end{equation*} 
if, and only if,
\begin{equation*}
|a_1|^2+|a_2|^2+|a_4|^2+|a_5|^2+|a_6|^2+|a_7|^2+|a_9|^2+|a_{10}|^2+|a_{11}|^2=2\mathfrak{Re}(a_3\overline{a}_8+a_3\overline{a}_{12}+a_8\overline{a}_{12}).
\end{equation*}
\begin{rem}\label{hol_par_astheno}
If the complex manifold $(M,J)$ is holomorphically parallelizable, i.e., 
\[
a_3=a_4=a_5=a_7=a_8=a_9=a_9=a_{10}=a_{11}=a_{12}=0,
\]
then metric $g$ is astheno K\"ahler on $(M,J)$ if, and only if, also $a_3=a_8=a_{12}=0$, i.e., $(M,J)$ is a complex torus. This is in line with the more general argument that on a compact holomorphically parallelizable manifold there exists a global coframe of holomorphic $(1,0)$-form; however, if a manifold admits an astheno-K\"ahler metric, every holomorphic $1$-form is $d$-closed. Therefore, on a holomorphically parallelizable manifold endowed with an astheno-K\"ahler metric, each form of the global holomorphic coframe is $d$-closed, hence the manifold is a torus.
\end{rem}
Let $\{Z_1,Z_2,Z_3,Z_4\}$ be the dual frame of $\{\eta^1,\eta^2,\eta^3,\eta^4\}$ and define the smooth $(0,1)$-vector form $\phi(\mathbf{t})$ on $(M,J)$ by
\begin{equation}\label{eq:vector_def_astheno}
\phi(\mathbf{t}):=t_1\overline{\eta}^1\otimes Z_1+t_2\overline{\eta}^2\otimes Z_2+t_3\overline{\eta}^3\otimes Z_3, \qquad \mathbf{t}\in B:=\{(t_1,t_2,t_3)\in\C^3: |t_j|<1\},
\end{equation}
which parametrizes a family of (non necessarily integrable) deformations $\{(M,J_t)\}_{t\in B}$ of $(M,J)$. By Lemma \ref{lemma:rao-zhao}, each (almost) complex structure $J_t$ can be characterized by the coframe $\{\eta_{\mathbf{t}}^1,\eta_{\mathbf{t}}^2,\eta_{\mathbf{t}}^3,\eta_{\mathbf{t}}^4\}$ given by
\begin{align}\label{eq:struct_eq_def_astheno_1}
\begin{cases}
\,&\eta_{\mathbf{t}}^1=\eta^1+t_1\overline{\eta}^1 \vspace{0.2cm}\\
&\eta_{\mathbf{t}}^2=\eta^2+t_2\overline{\eta}^2 \vspace{0.2cm}\\
&\eta_{\mathbf{t}}^3=\eta^3+t_3\overline{\eta}^3 \vspace{0.2cm}\\
&\eta_{\mathbf{t}}^4=\eta^4,
\end{cases}
\end{align}
which yields, by inverting the system,
\begin{align}\label{eq:struct_eq_def_astheno_2}
\begin{cases}
\,&\eta^1=\frac{1}{1-|t_1|^2}(\eta_{\mathbf{t}}^1-t_1\overline{\eta}_{\mathbf{t}}^1) \vspace{0.3cm}\\
&\eta^2=\frac{1}{1-|t_2|^2}(\eta_{\mathbf{t}}^2-t_2\overline{\eta}_{\mathbf{t}}^2) \vspace{0.3cm}\\
&\eta^3=\frac{1}{1-|t_3|^2}(\eta_{\mathbf{t}}^3-t_3\overline{\eta}_{\mathbf{t}}^3) \vspace{0.3cm}\\
&\eta^4=\eta_{\mathbf{t}}^4.
\end{cases}
\end{align}
Set $T_j:=\frac{1}{1-|t_j|^2}$, for $j\in\{1,2,3\}$. 
As recalled in the previous section, $\phi(\mathbf{t})$ parametrizes a family of compact complex manifold, i.e., each $J_t$ defines an integrable almost complex structure on $M$, if, and only if,
\begin{equation*}
(d\eta_{\mathbf{t}}^j)^{0,2}=0, \qquad j\in\{1,2,3,4\}.
\end{equation*}
By relations \eqref{eq:struct_eq_def_astheno_1}, \eqref{eq:struct_eq_def_astheno_2} and structure equations \eqref{eq:struct_eq_def_astheno_0}, it turns out that such an integrability condition  holds if, and only if
\begin{align*}
(d\eta_{\mathbf{t}}^4)^{0,2}&= T_1T_2(a_1t_1t_2-a_4t_1+a_7t_2)\eta_{\mathbf{t}}^{\overline{12}}\\
&+ T_1T_2(a_2t_1t_3-a_5t_1+a_{10}t_3)\eta_{\mathbf{t}}^{\overline{13}}\\
&+T_2T_3(a_6t_2t_3-a_9t_2+a_{11}t_3)\eta_{\mathbf{t}}^{\overline{23}}=0,
\end{align*}
i.e.,
\begin{align}\label{eq:integrability_def_astheno_nilp}
\begin{cases}
\, &a_1t_1t_2-a_4t_1+a_7t_2=0\\
&a_2t_1t_3-a_5t_1+a_{10}t_3=0\\
&a_6t_2t_3-a_9t_2+a_{11}t_3=0.
\end{cases}
\end{align}
Let us now fix a choice of $a_{j}\in\q[i]$ and let $S\subset B$ be the set of solutions of system \eqref{eq:integrability_def_astheno_nilp}. Then, for every $\mathbf{t}\in S$, $\phi(\mathbf{t})$ parametrizes a family of deformations $\{(M,J_{\mathbf{t}})\}_{\mathbf{t}\in S}$. We can construct a curve of deformations $\gamma\colon(-\epsilon,\epsilon)\rightarrow S$, with $\gamma(t):=(\phi_1(t),\phi_2(t),\phi_3(t))\in S$, 
so that, for every $t\in(-\epsilon,\epsilon)$, the $(0,1)$-vector form
\[
\phi(\gamma(t))=\phi_1(t)\overline{\eta}^{1}\otimes Z_1 + \phi_2(t)\overline{\eta}^2\otimes Z_2 + \phi_3(t)\overline{\eta}^3\otimes Z_3
\]
parametrizes a curve of deformations of $(M,J)$. (From now on, with an abuse of notation, we will write  $\phi(t):=\phi(\gamma(t))$.) We will then denote
\[
\phi'(0):=\phi_1'(0)\overline{\eta}^1\otimes Z_1+\phi_2'(0)\overline{\eta}^2\otimes Z_2 + \phi_2'(0)\overline{\eta}^3\otimes Z_3.
\]
We observe the following facts.
\begin{lemma}\label{lem:tecnico_astheno_def}
\begin{enumerate}
\item $(\del\circ  \iota_{\phi'(0)}\circ\del)\omega^2=C_{J,\phi'(0)}\eta^{123\overline{123}}$.
\item If $J$ is abelian, then $(\del\circ \iota_{\phi'(0)}\circ\del) \omega^2=0$.
\item The Bott-Chern cohomology class $[\eta^{123\overline{123}}]_{BC}\neq 0$ if, and only if, $\omega$ is SKT.
\end{enumerate}
\end{lemma}
\begin{proof}[Proof]
$\mathit{(1).}$ By simple computations, it holds that
\begin{equation*}
(-2)\omega^2=\eta^{1\overline{1}2\overline{2}}+\eta^{1\overline{1}3\overline{3}}+\eta^{1\overline{1}4\overline{4}}+\eta^{2\overline{2}3\overline{3}}+\eta^{2\overline{2}4\overline{4}}+\eta^{3\overline{3}4\overline{4}}.
\end{equation*}
Let us rewrite structure equations \eqref{eq:struct_eq_def_astheno_0} as
\begin{align*}
\begin{cases}
d\eta^j=0, \quad j\in\{1,2,3\},\vspace{0.2cm}\\
d\eta^4=\sum_{1\leq i<j\leq 3}A_{ij}\eta^{ij}+\sum_{i,j=1}^3B_{i\overline{j}}\eta^{i\overline{j}}.
\end{cases}
\end{align*}
For the sake of completeness, we write
\begin{align*}
\del\eta^4&=\sum_{1\leq i<j\leq 3}A_{ij}\eta^{ij}, & \delbar\eta^4&=\sum_{i,j=1}^3B_{i\overline{j}}\eta^{i\overline{j}}\\
\del\overline{\eta}^4&=-\sum_{i,j=1}^3\overline{B}_{i\overline{j}}\eta^{j\overline{i}}, & \delbar\overline{\eta}^4&=\sum_{1\leq i<j\leq 3}\overline{A}_{ij}\eta^{\overline{ij}}.
\end{align*}
Then, it is easy to see that
\begin{align*}
(-2)\del\omega^2&=(\eta^{1\overline{1}}+\eta^{2\overline{2}}+\eta^{3\overline{3}})\wedge\del\eta^{4\overline{4}}\\
&=(\eta^{1\overline{1}}+\eta^{2\overline{2}}+\eta^{3\overline{3}})\wedge(A_{ij}\eta^{ij\overline{4}}+\overline{B}_{i\overline{j}}\eta^{4j\overline{i}})\\
&=(\eta^{1\overline{1}}+\eta^{2\overline{2}}+\eta^{3\overline{3}})\wedge(A_{ij}\eta^{ij\overline{4}})\\
&+(\eta^{1\overline{1}}+\eta^{2\overline{2}}+\eta^{3\overline{3}})\wedge(\overline{B}_{i\overline{j}}\eta^{4j\overline{i}})
\end{align*}
Now, if $\phi'(0)=\phi_1'(0)\overline{\eta}^1\otimes Z_1+\phi_2'(0)\overline{\eta}^2\otimes Z_2 +\phi_3'(0)\overline{\eta}^3\otimes Z_3$, we have that
\begin{align*}
(-2)(\iota_{\phi'(0)}\circ\del)\omega^2&=(\eta^{1\overline{1}}+\eta^{2\overline{2}}+\eta^{3\overline{3}})\wedge\, \left[A_{ij}\left(\sum_{k=1}^3\phi_k'(0)\,\overline{\eta}^k\wedge i_{Z_k}(\eta^{ij})\right)\right]\wedge\overline{\eta}^4\\
&-(\eta^{1\overline{1}}+\eta^{2\overline{2}}+\eta^{3\overline{3}})\wedge \eta^4\wedge\left[\overline{B}_{i\overline{j}}\left( \sum_{k=1}^3\phi_k'(0)\,\overline{\eta}^k\wedge i_{Z_k}(\eta^{j\overline{i}})\right)\right].
\end{align*}
Note that the $(1,1)$-forms 
\[
\Omega_1:=A_{ij}\left(\sum_{k=1}^3\phi_k'(0)\,\overline{\eta}^k\wedge i_{Z_k}(\eta^{ij})\right)
\]
and
\[
\Omega_2:=\overline{B}_{i\overline{j}}\left( \sum_{k=1}^3\phi_k'(0)\,\overline{\eta}^k\wedge i_{Z_k}(\eta^{j\overline{i}})\right)
\]
do not contain $\eta^4$ nor $\overline{\eta}^4$. Then,
\begin{align*}
(-2)(\del\circ \iota_{\phi'(0)})\circ\del)(\omega^2)&=(\eta^{1\overline{1}}+\eta^{2\overline{2}}+\eta^{3\overline{3}})\wedge\Omega_1\wedge\del\overline{\eta}^4\\
&+(\eta^{1\overline{1}}+\eta^{2\overline{2}}+\eta^{3\overline{3}})\wedge\del\eta^4\wedge\Omega_2\\
&=-(\eta^{1\overline{1}}+\eta^{2\overline{2}}+\eta^{3\overline{3}})\wedge\Omega_1\wedge
\left(\overline{B}_{i\overline{j}}\eta^{j\overline{i}}\right)\\
&+(\eta^{1\overline{1}}+\eta^{2\overline{2}}+\eta^{3\overline{3}})\wedge\left( A_{ij}\eta^{ij}\right)\wedge\Omega_2,
\end{align*}
i.e, the form $\del\circ \iota_{\phi'(0)}\circ \del \omega^2$ is a $(3,3)$-form on $(M,J)$ which does not contain $\eta^4$ nor $\overline{\eta}^4$. Hence, $(\del\circ \iota_{\phi'(0)}\circ \del)\omega^2=C_{J,\phi'(0)}\eta^{123\overline{123}}$, where $C_{J,\phi'(0)}\in \C$ is a constant depending on the the structure equations defining the complex structure $J$ and the derivative of $\phi(t)$ in $t=0$.\vspace{0.3cm}\\
$\mathit{(2).}$ If $J$ is abelian, then $A_{ij}=0$, for every $i,j$ with $1\leq i<j\leq 3$. Then, by the previous point, $\Omega_1=0$ and clearly $(\del\circ \iota_{\phi'(0)}\circ \del)\omega^2=0$.\vspace{0.3cm}\\
$\mathit{(3).}$ We note that since the form $\eta^{123\overline{123}}$ is $d$-closed, the Bott-Chern class $[\eta^{123\overline
{123}}]$ is well defined. Let us assume that the metric $\omega$ is SKT, i.e., $\del\delbar\omega=0$. By structure equations \eqref{eq:struct_eq_def_astheno_0} this is equivalent to $\del\delbar\eta^{4\overline{4}}=0$. S Moreover, since $\del\delbar\ast_g\eta^{123\overline{123}}=\del\delbar\eta^{4\overline{4}}=0$, the form $\eta^{123\overline{123}}$ is Bott-Chern harmonic, hence $[\eta^{123\overline{123}}]_{BC}\neq 0$.\\
Viceversa, let us assume that $\omega$ is not SKT, i.e, $\del\delbar\omega\neq0$, which, by structure equations, is equivalent to $\del\delbar\eta^{4\overline{4}}\neq 0$. Note that $\del\delbar\eta^{4\overline{4}}$ is $(2,2)$-form on $M$, hence $\del\delbar\eta^{4\overline{4}}=\sum_{i<j,k<l}A_{ij\overline{kl}}\eta^{ij\overline{kl}}$, with $i,j,k,\in\{1,2,3\}$. Then, one can choose an invariant $(1,1)$-form $\alpha$ on $M$ such that it does contain $\eta^4$ nor $\overline{\eta}^4$ and $\del\delbar\eta^{4\overline{4}}\wedge\alpha=C\eta^{123\overline{123}}$, for a constant $0\neq C\in\C$. In particular, $\del\alpha=\delbar\alpha=0$. But then,
\begin{equation}
\del\delbar(\frac{1}{C}\eta^{4\overline{4}}\wedge\alpha)=\frac{1}{C}\del\delbar(\eta^{4\overline{4}})\wedge\alpha=\eta^{123\overline{123}},
\end{equation}
which implies that the form $\eta^{123\overline{123}}$ is $\del\delbar$-exact. Therefore $[\eta^{123\overline{123}}]_{BC}= 0$.\vspace{0.3cm}
\end{proof}

\begin{rem}\label{rem:lemma_tecnico_4_dim}
As a consequence of Lemma \ref{lem:tecnico_astheno_def} and Remark \ref{hol_par_astheno}, it turns out that Corollary \ref{cor:def_curve_astheno} can yield obstructions on this family of $4$-dimensional nilmanifolds along the curve of deformations parametrized by $\phi(t)$, if the starting complex structure $J$ is not holomorphically parallelizable nor abelian and the starting diagonal metric on $(M,J)$ is both astheno-K\"ahler and SKT.
\end{rem}
For computations in higher dimensions, it is rather easy to see that, following Remark \ref{hol_par_astheno}, Lemma \ref{lem:tecnico_astheno_def}, and Remark \ref{rem:lemma_tecnico_4_dim}, similar arguments are valid also for any $n$-dimensional nilmanifold $(M,J)$ with left-invariant complex structure $J$ characterized by analogous structure equations, i.e., it holds the following.
\begin{lemma}
Let $(M=\Gamma\backslash G,J)$ be a nilmanifold whose complex structure is determined by a base $\{\eta^1,\dots,\eta^n\}$ of the complexified dual of the Lie algebra $\mathfrak{g}$ of $G$ such that
\begin{equation*}
\begin{cases}
d\eta^i=0, \quad i\in\{1,\dots n-1\},\\
d\eta^n\in\Span_{\C}\{\eta^{ij},\eta^{k\overline{l}}\}, \quad \,\, i,j,k,l\in\{1,\dots, n-1\},
\end{cases}
\end{equation*}
and whose structure constants are elements of $\q[i]$. More specifically, if $\omega=\frac{i}{2}\sum_{j=1}^n\eta^{j\overline{j}}$ is the fundamental form associated to the diagonal metric $g$ and $\phi(t)=\sum_{k=1}^n\phi_{k}(t)\overline{\eta}^k\otimes Z_k$, for $t\in(-\epsilon,\epsilon)$ is a curve of deformations of $(M,J)$, then
\begin{enumerate}
\item $(\del\circ \iota_{\phi'(0)}\circ \del)\omega^{n-2}=C_{J,\phi'(0)}\,\eta^{1\dots n-1\,\overline{1}\dots \overline{n-1}}$.
\item If $J$ is abelian, then $(\del\circ \iota_{\phi'(0)}\circ \del) \omega^{n-2}=0$.
\item The Bott-Chern cohomology class $[\eta^{1\dots n-1\, \overline{1}\dots \overline{n-1}}]\neq 0$ if and only if $\omega$ is SKT.
\end{enumerate}
\end{lemma}
As in the $4$-dimensional case, the existence of curves of astheno-K\"ahler metrics along the family of deformations parametrized by $\phi$ could be obstructed by Corollary \ref{cor:def_curve_astheno} if the canonical diagonal metric $g$ is also SKT and the complex structure $J$ is neither abelian nor holomorphically parallelizable.

In the $4$-dimensional example, let us consider an element $(M,J)$ of the family of $4$-dimensional nilmanifolds, with $a_2=a_5=a_6=a_9=a_{10}=a_{11}=a_{12}=0$. In this case, structure equations become
\begin{align}\label{struct_eq_def_astheno_I}
\begin{cases}
d\eta^i&=0,\quad i\in\{1,2,3\},\\
d\eta^4&=a_1\eta^{12}+a_3\eta^{1\overline{1}}+a_4\eta^{1\overline{2}}\\
&\,\,+a_7\eta^{2\overline{1}}+a_8\eta^{2\overline{2}}.
\end{cases}
\end{align}
Note that the manifold $(M,J)$ can be realized as the direct product of a complex $1$-dimensional torus and a $3$-dimensionl nilmanifold.
 
The diagonal metric $g$ with fundamental form $\omega=\frac{i}{2}\sum_{j=1}^4\eta^{j\overline{j}}$ is astheno-K\"ahler if, and only if, it is SKT if, and only if,
\begin{equation}\label{eq:astheno_cond_I}
|a_1|^2+|a_4|^2+|a_7|^2=2\mathfrak{Re}(a_3\overline{a}_8).
\end{equation} 
We then consider the $(0,1)$-vector form $\phi(\mathbf{t})$ on $(M,J)$, $\mathbf{t}\in B$, as in \eqref{eq:vector_def_astheno}. Such a vector form parametrizes a family of deformations $\{(M,J_{\mathbf{t})}\}_{\mathbf{t}\in B}$ of $(M,J)$. Moreover, each $J_{\mathbf{t}}$ is an integrable almost complex structure on $M$ if $\mathbf{t}=(t_1,t_2,t_3)$ satisfies
\begin{equation}\label{eq:def_par_I}
a_1t_1t_2-a_4t_1+a_7t_2=0.
\end{equation}
Let us set $F(t_1,t_2,t_3):=a_1t_1t_2-a_4t_1+a_7t_2$, $(t_1,t_2,_3)\in B$. Then, the gradient of $F$ in $(0,0,0)$ is
\begin{equation}
\nabla F\restrict{(0,0,0)}=\begin{pmatrix}
-a_4\\ a_7 \\ 0
\end{pmatrix}.
\end{equation}
By distinguishing the cases in which either $\nabla F\restrict{(0,0,0)}=0$ or $\nabla F\restrict{(0,0,0)}\neq 0$, we obtain the following.
\subsubsection{Case (i): $\nabla F\restrict{(0,0,0)}=0.$} In this case, it holds that $a_4=a_7=0$. Hence the only non trivial equation of \eqref{struct_eq_def_astheno_I} becomes
\begin{equation*}
d\eta^{4}=a_1\eta^{12}+a_3\eta^{1\overline{1}}+a_8\eta^{2\overline{2}},
\end{equation*}
and let assume that $\omega$ is astheno-K\"ahler (and SKT), i.e., it holds
\begin{equation*}
|a_1|^2=2\mathfrak{Re}(a_3\overline{a}_8).
\end{equation*}
We will assume $a_1\neq 0$.

The solution set of \eqref{eq:def_par_I} is
\begin{equation*}
S=\{(t_1,t_2,t_3)\in B: t_1t_2=0\}.
\end{equation*} 
Therefore, the $(0,1)$-vector form $\phi(\mathbf{t})$ parametrizing the integrable deformations of $M$ is
\[
\phi(\mathbf{t})=t_1\overline{\eta}^1\otimes Z_1+t_2\overline{\eta}^2\otimes Z_2 +t_3\overline{\eta}^3\otimes Z_3, \quad (t_1,t_2,t_3)\in S,
\]
and we can then consider the curve
\begin{equation*}
\phi(t)=t\cdot u_1\overline{\eta}^1\otimes Z_1 +t\cdot u_2\overline{\eta}^2\otimes Z_2 +t\cdot u_3 \overline{\eta}^3\otimes Z_3,\quad  t\in(-\epsilon,\epsilon),
\end{equation*}
for a fixed $(u_1,u_2,u_3)\in S$ and sufficiently small $\epsilon>0$. Hence, 
\begin{equation*}
\phi'(0)=u_1\overline{\eta}^1\otimes Z_1 + u_2\overline{\eta}^2\otimes Z_2 + u_3 \overline{\eta}^3\otimes Z_3.
\end{equation*}
We can then apply our condition and compute
\begin{equation*}
(\del\circ \iota_{\phi'(0)}\circ \del)\omega^2=0,
\end{equation*}
hence neither Corollary \ref{cor:def_curve_astheno} and Theorem \ref{thm:def_curve_astheno} yield obstructions.

\subsubsection{Case (ii): $\nabla F\restrict{(0,0,0)}\neq 0$.}

In this case, it holds that either $a_4\neq 0$ or $a_7\neq 0$. Suppose that $a_4\neq 0$.

Then, structure equations \eqref{struct_eq_def_astheno_I} and the astheno-K\"ahler (and SKT) condition for $\omega$ \eqref{eq:astheno_cond_I} still holds. Then solution set of \eqref{eq:def_par_I} is then
\begin{equation*}
S=\left\{(t_1,t_2,t_3)\in\C^3:\,\, t_1=\frac{a_7t_2}{a_4-a_1t_2},\,\,\,|t_2|<\delta\right\},
\end{equation*}
for $\delta>0$ sufficiently small.

Hence the $(0,1)$-vector form $\phi(\mathbf{t})$ parametrizing the integrable deformations of $(M,J)$ is
\begin{equation*}
\phi(\mathbf{t})=\frac{a_7t_2}{a_4-a_1t_2}\overline{\eta}^1\otimes Z_1 + t_2\overline{\eta}^2\otimes Z_2 +t_3\overline{\eta}^3\otimes Z_3, \quad \left(\frac{a_7t_2}{a_4-a_1t_2},t_2,t_3\right)\in S,
\end{equation*}
so that can consider the curve of deformations
\begin{equation*}
\phi(t)=\frac{ta_7u_2}{a_4-ta_1u_2}\overline{\eta}^1\otimes Z_1 +tu_2\overline{\eta}^2\otimes Z_2 + tu_3\overline{\eta}^3\otimes Z_3,\quad  \left(u_2,u_3\right)\in \C^2,\,\, t\in(-\epsilon,\epsilon),
\end{equation*}
for $\epsilon>0$ sufficiently small, so that
\begin{equation*}
\phi'(0)=\frac{a_7u_2}{a_4}\overline{\eta}^1\otimes Z_1 +u_2\overline{\eta}^2\otimes Z_2+u_3\otimes Z_3.
\end{equation*}
By direct computations, we obtain that
\begin{equation*}
\del\circ \iota_{\phi'(0)}\circ \del \omega^2=2(|a_7|^2-|a_4|^2)\frac{a_1u_2}{a_4}\eta^{123\overline{123}}.
\end{equation*}
Therefore, since the Bott-Chern cohomology class $[\eta^{123\overline{123}}]\neq 0$ by Lemma \ref{lem:tecnico_astheno_def}, condition \eqref{eq_cor_def_curve_astheno} holds if and only if
\begin{equation*}
\left(|a_4|^2-|a_7|^2\right)\mathfrak{Re}\left(\frac{a_1u_2}{a_4}\right)=0.
\end{equation*}
Let us consider the case $a_7\neq 0$. We have that the solution set of \eqref{eq:def_par_I} is
\begin{equation*}
S=\left\{(t_1,t_2,t_3)\in B:\, t_2=\frac{a_4t_1}{a_7+a_1t_1},\,\,\,|t_1|<\delta\right\},
\end{equation*}
for $\delta>0$ sufficiently small.

We can then pick as a curve of deformations
\begin{equation*}
\phi(t)=tu_1\overline{\eta}^1\otimes Z_1+\frac{ta_4u_1}{a_7+ta_1u_1}\overline{\eta}^2\otimes Z_2+ tu_3\overline{\eta}^3\otimes Z_3, \quad \left(u_1, u_3\right)\in \C^2,\,\, t\in(-\epsilon,\epsilon),
\end{equation*}
for $\epsilon>0$ sufficiently small, so that
\begin{equation*}
\phi'(0)=u_1\overline{\eta}^1\otimes Z_1+\frac{a_4u_1}{a_7}\overline{\eta}^2\otimes Z_2 +u_3\overline{\eta}^3\otimes Z_3.
\end{equation*}
Then, by computations similar to the previous case we obtain that
\begin{equation*}
[\mathfrak{Im}(\del\circ \iota_{\phi'(0)}\circ\del (\omega^2))]_{BC}=0\iff \left(|a_4|^2-|a_7|^2\right)\mathfrak{Re}\left(\frac{a_1u_1}{a_7}\right)=0.
\end{equation*}
By applying Corollary \ref{cor:def_curve_astheno} to each case, we obtain the following theorem.
\begin{theorem}
Let $(M=\Gamma\backslash G,J)$ be an element of the family of $4$-dimensional nilmanifolds with complex structure $J$ determined by the covectors $\{\eta^1,\eta^2,\eta^3,\eta^4\}$ on the complexified dual of the Lie algebra $\mathfrak{g}$ of $G$, with structure equations
\begin{align*}
\begin{cases}
d\eta^i=0, \quad i\in\{1,2,3\}\\
d\eta^4=a_1\eta^{12}+a_3\eta^{1\overline{1}}+a_4\eta^{1\overline{2}}+a_7\eta^{2\overline{1}}+a_8\eta^{2\overline{2}},
\end{cases}
\end{align*}
with $a_1,a_3,a_4,a_7,a_8\in\q[i]$. Let $\omega=\frac{i}{2}\sum_{j=1}\eta^{j\overline{j}}$ be the fundamental form associated to the diagonal metric, which we assume to be astheno-K\"ahler, i.e.,
\begin{equation*}
|a_1|^2+|a_4|^2+|a_7|^2=2\mathfrak{Re}(a_3\overline{a}_8).
\end{equation*}
Then,
\begin{itemize}
\item if $a_4\neq 0$ and $(u_2,u_3)\in\C^2$, there exists no curve of astheno-K\"ahler metrics $\omega_t$ such that $\omega_0=\omega$ along the curve of deformations $t\mapsto\phi(t)=\frac{ta_7u_2}{a_4-ta_1u_2}\overline{\eta}^1\otimes Z_1 +tu_2\overline{\eta}^2\otimes Z_2 + tu_3\overline{\eta}^3\otimes Z_3$, for $t\in(-\epsilon,\epsilon)$ if
\begin{equation*}
\left(|a_4|^2-|a_7|^2\right)\mathfrak{Re}\left(\frac{a_1u_2}{a_4}\right)\neq 0.
\end{equation*}
\item if $a_7\neq 0$ and $(u_1,u_3)\in\C^2$, there exists no curve of astheno-K\"ahler metrics $\omega_t$ such that $\omega_0=\omega$ along the curve of deformations $t\mapsto\phi(t)=u_1\overline{\eta}^1\otimes Z_1+\frac{ta_4u_1}{a_7+ta_1u_1}\overline{\eta}^2\otimes Z_2+ tu_3\overline{\eta}^3\otimes Z_3$, for $t\in(-\epsilon,\epsilon)$, if
\begin{equation*}
\left(|a_4|^2-|a_7|^2\right)\mathfrak{Re}\left(\frac{a_1u_1}{a_7}\right)\neq 0.
\end{equation*}
\end{itemize}
\end{theorem}

\subsection{Example 2}
We now show an application of Corollary \ref{cor:def_curve_astheno} to a different family of $4$-dimensional nilmanifolds with invariant complex structure. Let $(M=\Gamma\backslash G,J)$ be the nilmanifold with complex structure $J$ determined by the base $\{\eta^1,\eta^2,\eta^3,\eta^4\}$ of the complexified dual $\mathfrak{g}_{\C}^*$ of the Lie algebra $\mathfrak{g}$ of $G$, with structure equations
\begin{equation}
\begin{cases}
d\eta^1=0, \vspace{0.1cm}\\ 
d\eta^2=0, \vspace{0.1cm}\\
d\eta^3=a_1\eta^{12}+a_2\eta^{1\overline{1}}+a_3\eta^{1\overline{2}}+a_4\eta^{2\overline{1}}+a_5\eta^{2\overline{2}}\vspace{0.1cm}\\
d\eta^4=b_1\eta^{12}+b_2\eta^{1\overline{1}}+b_3\eta^{1\overline{2}}+b_4\eta^{2\overline{1}}+b_5\eta^{2\overline{2}}
\end{cases}
\end{equation}
with $a_j,b_j\in\q[i]$. As in Example \ref{example1}, the complex structure $J$ is $2$-step nilpotent. However, the manifold $(M,J)$ does not have the structure of a direct product of a $1$-dimensional complex torus and a $3$-dimensional nilmanifold.

Let $\omega=\sum_{j=1}^4\eta^{j\overline{j}}$ be the fundamental form associated to the diagonal metric $g$. By direct computations, it turns out that $g$ is astheno-K\"ahler if and only if
\begin{equation}\label{eq:astheno_cond_II}
\begin{cases}
2\mathfrak{Re}(b_5\overline{b}_2)-|b_1|^2-|b_3|^2-|b_4|^2=0 \vspace{0.1cm}\\
|b_1|^2+|b_3|^2-|b_4|^2-b_5\overline{a}_2-b_2\overline{a}_5=0 \vspace{0.1cm}\\
2\mathfrak{Re}(a_5\overline{a}_2)-|b_1|^2-|b_3|^2-|b_4|^2=0.
\end{cases}
\end{equation}
Since $\del\delbar\omega=0$ if and only if 
\begin{equation*}
\mathfrak{Re}(a_5\overline{a}_2)+\mathfrak{Re}(b_5\overline{b}_2)-|b_1|^2-|b_3|^2-|b_4|^2=0,
\end{equation*}
if the metric $g$ is astheno-K\"ahler, it is also SKT.

Let us consider the family of deformations $\{(M,J_{\mathbf{t}})\}_{t\in B}$ of $(M,J)$ parametrized by the $(0,1)$-vector form
\begin{equation*}
\phi(\mathbf{t})=t_1\overline{\eta}^1\otimes Z_1 + t_2\overline{\eta}^{2}\otimes Z_2,
\end{equation*}
with $\mathbf{t}=(t_1,t_2)\in B:=\{\mathbf{t}\in\C^2:\, |t_j|<1,\,\,\, j\in\{1,2\}\}$. The coframe $\{\eta^1,\eta^2,\eta^3,\eta^4\}$ then changes under $\phi(\mathbf{t})$ as
\begin{equation*}
\begin{cases}
\eta_{\mathbf{t}}^1=\eta^1-t_1\overline{\eta}^1\\
\eta_{\mathbf{t}}^2=\eta^2-t_2\overline{\eta}^2\\
\eta_{\mathbf{t}}^3=\eta^3\\
\eta_{\mathbf{t}}^4=\eta^4
\end{cases}
\end{equation*}
so that, by inverting the system, we obtain
\begin{equation*}
\begin{cases}
\eta^1=\frac{1}{1-|t_1|^2}(\eta_{\mathbf{t}}^1-t_1\overline{\eta}_{\mathbf{t}}^1) \vspace{0.2cm}\\
\eta^2=\frac{1}{1-|t_2|^2}(\eta_{\mathbf{t}}^2-t_2\overline{\eta}_{\mathbf{t}}^2) \vspace{0.2cm}\\
\eta^3=\eta_{\mathbf{t}}^3 \vspace{0.2cm}\\
\eta^4=\eta_{\mathbf{t}}^4.
\end{cases}
\end{equation*}
Set $T_j:=\frac{1}{1-|t_j|^2}$, for $j\in\{1,2\}$.

Since the form $\phi(\mathbf{t})$ defines an family of complex manifolds if and only if $d(\eta_{\mathbf{t}}^j)^{0,2}=0$, for every $j\in\{1,2,3,4\}$, such a integrability condition is satisfied if and only if $(d\eta_{\mathbf{t}}^3)^{0,2}=0$, which yields
\begin{align*}
T_1T_2(a_1t_1t_2-a_3t_1+a_4t_2)\eta_{\mathbf{t}}^{\overline{12}}=0
\end{align*}
and $(d\eta_{\mathbf{t}}^4)^{0,2}=0$, which yields
\begin{align*}
T_1T_2(b_1t_1t_2-b_3t_1+b_4t_2)\eta_{\mathbf{t}}^{\overline{12}}=0.
\end{align*}
Under the assumption that $a_1=b_1$, $a_3=b_3$, and $a_4=b_4$, we have that the condition of integrability is valid for $t\in S$, where $S$ is the solution set of equation
\begin{equation}\label{eq:integrability_t_II}
b_1t_1t_2-b_3t_1+b_4t_2=0, \quad (t_1,t_2)\in B.
\end{equation}
We now proceed as in the previous example, by considering the map
\begin{equation*}
F(t_1,t_2)=b_1t_1t_2-b_3t_1+b_4t_2,
\end{equation*}
and discussing the cases in which either $\nabla F\restrict{(0,0)}=\begin{pmatrix}
-b_3\\ b_4
\end{pmatrix}$ vanishes or not.
\subsubsection{Case (i): $\nabla F\restrict{(0,0)}= 0$}
This is the situation in which $b_3=b_4=0$. Then structure equations become
\begin{equation}
\begin{cases}
d\eta^1=0, \vspace{0.1cm}\\ 
d\eta^2=0, \vspace{0.1cm}\\
d\eta^3=b_1\eta^{12}+a_2\eta^{1\overline{1}}+a_5\eta^{2\overline{2}}\vspace{0.1cm}\\
d\eta^4=b_1\eta^{12}+b_2\eta^{1\overline{1}}+b_5\eta^{2\overline{2}}
\end{cases}
\end{equation}
and the diagonal metric $g$ is astheno-K\"ahler if, and only if, the following condition holds
\begin{align*}
\begin{cases}
2\mathfrak{Re}(b_5\overline{b}_2)=|b_1|^2\vspace{0.1cm}\\
b_2\overline{a}_5+b_5\overline{a}_2=|b_1|^2\vspace{0.1cm}\\
2\mathfrak{Re}(a_5\overline{a}_2)=|b_1|^2
\end{cases}
\end{align*}
We assume that $b_1\neq 0$. The solution set $S$ for equation \eqref{eq:integrability_t_II} is then
\begin{equation*}
S=\{(t_1,t_2)\in B:t_1t_2=0\}
\end{equation*}
Hence as a curve of deformation $\phi(t)$ with $t\in S$ we can choose
\begin{equation*}
\phi(t)=tu_1\overline{\eta}^1\otimes Z_1+tu_2\overline{\eta}^2\otimes Z_2, \quad (u_1,u_2)\in S, \,\, t\in(-\epsilon,\epsilon)
\end{equation*}
for $\epsilon>0$ sufficiently small. Then,
\begin{equation*}
\phi'(0)=u_1\overline{\eta}^1\otimes Z_1+u_2\overline{\eta}^2\otimes Z_2.
\end{equation*}
By computations, however, it turns out that
\begin{equation*}
\del\circ \iota_{\phi'(0)}\circ \del \omega^2=0,
\end{equation*}
hence Corollary \ref{cor:def_curve_astheno} does not yield any obstruction.
\subsubsection{Case (ii): $\nabla F\restrict{(0,0)}\neq 0$.}
In this situation, either $b_3\neq 0$ or $b_4\neq 0$. Let us assume $b_3\neq 0$; the latter case is completely analogous. 

We have the following structure equations
\begin{equation*}
\begin{cases}
d\eta^1=0\\
d\eta^2=0\\
d\eta^3=b_1\eta^{12}+a_2\eta^{1\overline{1}}+b_3\eta^{1\overline{2}}+b_4\eta^{2\overline{1}}+a_5\eta^{2\overline{2}}\\
d\eta^4=b_1\eta^{12}+b_2\eta^{1\overline{1}}+b_3\eta^{1\overline{2}}+b_4\eta^{2\overline{1}}+b_5\eta^{2\overline{2}}
\end{cases}
\end{equation*} 
and the astheno-K\"ahler condition \eqref{eq:astheno_cond_II} on the diagonal metric $g$ is still valid. The solution set $S$ for \eqref{eq:integrability_t_II} is then
\begin{equation*}
S=\left\{(t_1,t_2)\in B\,:\,t_1=\frac{b_4t_2}{b_3-b_1t_2}, \,\,|t_2|<\delta \right\}
\end{equation*}
for $\delta>0$ sufficiently small. Once we fix $u_2\in\C$, we can pick the curve of deformations
\begin{equation*}
\phi(t)=\frac{tb_4u_2}{b_3-tb_1u_2}\overline{\eta}^1\otimes Z_1+tu_2\overline{\eta}^2\otimes Z_2, \quad t\in(-\epsilon,\epsilon),
\end{equation*}
for $\epsilon>0$ sufficiently small, so that
\begin{equation*}
\phi'(0)=\frac{b_4u_2}{b_3}\overline{\eta}^1\otimes Z_1+ u_2\overline{\eta}^2\otimes Z_2.
\end{equation*}
Then, we compute
\begin{align*}
(\del\circ \iota_{\phi'(0)}\circ \del)\omega^2=((|b_3|^2-|b_4|^2)\frac{b_1}{b_4}u_2)\eta^{123\overline{123}}+((|b_4|^2-|b_3|^2)\frac{b_1}{b_4}u_2)\eta^{123\overline{124}}\\
+((|b_4|^2-|b_3|^2)\frac{b_1}{b_4}u_2)\eta^{124\overline{123}}+((|b_3|^2-|b_4|^2)\frac{b_1}{b_4}u_2)\eta^{124\overline{124}}
\end{align*}
Now, the forms $\eta^{123\overline{123}}, \eta^{123\overline{124}},\eta^{124\overline{123}},\eta^{124\overline{124}}$ are all $d$-closed by structure equations. Moreover, by considering the $\C$-antilinear $\ast$-operator with respect to $g$, we see that
\begin{align*}
\del\delbar\ast\eta^{123\overline{123}}=0 &\iff 2\mathfrak{Re}(b_5\overline{b}_2)-|b_1|^2-|b_3|^2-|b_4|^2=0\\
\del\delbar\ast\eta^{123\overline{124}}=0 &\iff |b_1|^2+|b_3|^2+|b_4|^2-b_5\overline{a}_2-b_2\overline{a}_5=0\\
\del\delbar\ast\eta^{124\overline{123}}=0 &\iff |b_1|^2+|b_3|^2+|b_4|^2-a_2\overline{b}_5-a_5\overline{b}_2=0\\
\del\delbar\ast\eta^{124\overline{124}}=0&\iff 2\mathfrak{Re}(a_5\overline{a}_2)-|b_1|^2-|b_3|^2-|b_4|^2=0
\end{align*}
Since $\omega$ is astheno-K\"ahler, i.e., conditions \eqref{eq:astheno_cond_II} hold, hence $\del\delbar\ast\eta^{123\overline{123}}=\del\delbar\ast\eta^{123\overline{124}}=\del\delbar\ast\eta^{124\overline{123}}=\del\delbar\ast\eta^{124\overline{124}}=0$, i.e., the forms $\eta^{123\overline{123}}$, $\eta^{123\overline{124}}$, $\eta^{124\overline{123}}$, $\eta^{124\overline{124}}$. Therefore, 
\begin{align*}
[\mathfrak{Im}(\del\circ \iota_{\phi'(0)}\circ \del \omega^2)]_{BC}=&-i(|b_3|^2-|b_4|^2)\mathfrak{Re}(\frac{b_1u_2}{b_4})[\eta^{123\overline{123}}]_{BC}+i(|b_3|^2-|b_4|^2)\mathfrak{Re}(\frac{b_1u_2}{b_4})[	\eta^{123\overline{124}}]_{BC}\\
&+i(|b_3|^2-|b_4|^2)\mathfrak{Re}(\frac{b_1u_2}{b_4})[\eta^{124\overline{123}}]_{BC}-i(|b_3|^2-|b_4|^2)\mathfrak{Re}(\frac{b_1u_2}{b_4})[\eta^{124\overline{124}}]_{BC}
\end{align*}
which vanishes in $H_{BC}^{3,3}(M)$ if, and only if,
\begin{equation}
(|b_3|^2-|b_4|^2)\mathfrak{Re}(\frac{b_1u_2}{b_4})=0. 
\end{equation} 
We summarize what we obtained in the following theorem.
\begin{theorem}
Let $(M,J)$ be an element of the family of $4$-dimensional manifolds determined by the coframe of left-invariant $(1,0)$-forms $\{\eta^1,\eta^2,\eta^3,\eta^4\}$ with structure equations
\begin{equation*}
\begin{cases}
d\eta^1=0\\
d\eta^2=0\\
d\eta^3=b_1\eta^{12}+a_2\eta^{1\overline{1}}+b_3\eta^{1\overline{2}}+b_4\eta^{2\overline{1}}+a_5\eta^{2\overline{2}}\\
d\eta^4=b_1\eta^{12}+b_2\eta^{1\overline{1}}+b_3\eta^{1\overline{2}}+b_4\eta^{2\overline{1}}+a_5\eta^{2\overline{2}},
\end{cases}
\end{equation*} 
with $a_2,a_5,b_1,b_2,b_3,b_4,b_5\in\q[i]$. Let $\omega=\frac{i}{2}\sum_{j=1}^4\eta^{j\overline{j}}$ be the fundamental form associated to the diagonal metric $g$ and suppose that $g$ is astheno-K\"ahler (and hence SKT), i.e.,
\begin{equation*}
\begin{cases}
2\mathfrak{Re}(b_5\overline{b}_2)-|b_1|^2-|b_3|^2-|b_4|^2=0 \vspace{0.1cm}\\
|b_1|^2+|b_3|^2-|b_4|^2-b_5\overline{a}_2-b_2\overline{a}_5=0 \vspace{0.1cm}\\
2\mathfrak{Re}(a_5\overline{a}_2)-|b_1|^2-|b_3|^2-|b_4|^2=0.
\end{cases}
\end{equation*}
Then,
\begin{itemize}
\item if $b_3\neq 0$ and $u_2\in \C$, there exists no curve of astheno-K\"ahler metrics $\omega_t$ with $\omega_0=\omega$ along the curve of deformations $\phi(t)=\frac{tb_4u_2}{b_3-tb_1u_2}\overline{\eta}^1\otimes Z_1+tu_2\overline{\eta}^2\otimes Z_2$, $t\in(-\epsilon,\epsilon)$, if
\begin{equation*}
(|b_3|^2-|b_4|^2)\mathfrak{Re}(\frac{b_1u_2}{b_4})\neq 0.
\end{equation*}
\item if $b_4\neq 0$ and $u_1\in \C$, there exists no curve of astheno-K\"ahler metrics $\omega_t$ with $\omega_0=\omega$ along the curve of deformations $\phi(t)=tu_1\overline{\eta}^1\otimes Z_1+t\frac{tb_3u_1}{tb_1u_1+b_4}\overline{\eta}^2\otimes Z_2$, $t\in(-\epsilon,\epsilon)$, if
\begin{equation*}
(|b_3|^2-|b_4|^2)\mathfrak{Re}(\frac{b_1u_1}{b_3})\neq 0.
\end{equation*}
\end{itemize}
\end{theorem}

\end{document}